\documentclass[11pt,a4paper]{article}
\usepackage{fullpage}

\usepackage[utf8]{inputenc}
\usepackage{amsthm,amssymb}
\usepackage[leqno]{amsmath}
\usepackage{mathtools}
\usepackage{tcolorbox}
\usepackage{thmtools}
\usepackage{subcaption,thm-restate}
\usepackage[mathcal,mathscr]{eucal}
\usepackage{epsfig,graphicx,graphics,color}
\usepackage{caption}
\usepackage{enumerate}
\usepackage{enumitem, crossreftools}
\usepackage[sort]{cite}
\usepackage{titling}
\usepackage{hyperref}
\usepackage{tikz}
\hypersetup{
    colorlinks=true,       
    linkcolor=blue,        
    citecolor=red,         
    filecolor=magenta,     
    urlcolor=cyan,         
    linktocpage=true
}

\theoremstyle{plain}
\newtheorem{thm}{Theorem}
\newtheorem{lem}[thm]{Lemma}
\newtheorem{cor}[thm]{Corollary}
\newtheorem{cl}[thm]{Claim}

\newtheorem{prop}[thm]{Proposition}
\newtheorem{conj}[thm]{Conjecture}

\theoremstyle{definition}

\newenvironment{claimproof}[1][\proofname]
  {%
    \proof[#1]%
  }
  {%
    \endproof%
  }

\def\final{0}  
\def\iflong{\iffalse}
\ifnum\final=0  
\newcommand{\knote}[1]{{\color{red}[{\tiny \textbf{Kristóf:} \bf #1}]\marginpar{\color{red}*}}}
\newcommand{\tnote}[1]{{\color{blue}[{\tiny \textbf{Tamás:} \bf #1}]\marginpar{\color{blue}*}}}
\newcommand{\rnote}[1]{{\color{orange}[{\tiny \textbf{Reviewer:} \bf #1}]\marginpar{\color{orange}*}}}
\else 
\newcommand{\knote}[1]{}
\newcommand{\tnote}[1]{}
\newcommand{\rnote}[1]{}
\fi

\newcommand{\cI}{\mathcal{I}}
\newcommand{\cX}{\mathcal{X}}
\newcommand{\cY}{\mathcal{Y}}
\newcommand{\cB}{\mathcal{B}}

\newcommand{\cH}{\mathcal{H}}

\newcommand{\ul}[1]{\underline{\mathbf{#1}}}

\title{Exchange distance of basis pairs in split matroids}

\thanksmarkseries{alph}
\author{
Kristóf Bérczi\thanks{MTA-ELTE Momentum Matroid Optimization Research Group and MTA-ELTE Egerváry Research Group, Department of Operations Research, Eötvös Loránd University, Budapest, Hungary. Email: \texttt{kristof.berczi@ttk.elte.hu}.}
\and
Tamás Schwarcz\thanks{MTA-ELTE Momentum Matroid Optimization Research Group, Department of Operations Research, Eötvös Loránd University, Budapest, Hungary. Email: \texttt{tamas.schwarcz@ttk.elte.hu}.}
}

\date{}

\begin{document}
\maketitle

\begin{abstract}
The basis exchange axiom has been a driving force in the development of matroid theory. However, the axiom gives only a local characterization of the relation of bases, which is a major stumbling block to further progress, and providing a global understanding of the structure of matroid bases is a fundamental goal in matroid optimization. 

While studying the structure of symmetric exchanges, Gabow proposed the problem that any pair of bases admits a sequence of symmetric exchanges. A different extension of the exchange axiom was proposed by White, who investigated the equivalence of compatible basis sequences. These conjectures suggest that the family of bases of a matroid possesses much stronger structural properties than we are aware of.

In the present paper, we study the distance of basis pairs of a matroid in terms of symmetric exchanges. In particular, we give a polynomial-time algorithm that determines a shortest possible exchange sequence that transforms a basis pair into another for split matroids, a class that was motivated by the study of matroid polytopes from a tropical geometry point of view. As a corollary, we verify the above mentioned long-standing conjectures for this large class. Being a subclass of split matroids, our result settles the conjectures for paving matroids as well.

\medskip

\end{abstract}

\section{Introduction}

Matroids are distinguished from ordinary hereditary set systems by their strong structural properties. There are several equivalent ways to define a matroid axiomatically. In terms of bases, a matroid $M$ is a pair $(S,\mathcal{B})$ where $S$ is a \textbf{ground set} and $\mathcal{B}\subseteq 2^S$ is a family of subsets satisfying the \textbf{basis axioms}: (B1) $\mathcal{B}\neq\emptyset$, (B2) for any $A,B\in \mathcal{B}$ and $e\in A- B$ there exists $f\in B- A$ such that $A-e+f\in\mathcal{B}$. Property (B2) is called the \textbf{exchange axiom} and serves as a distinctive feature of matroids. The exchange axiom has a mirror version, the \textbf{co-exchange axiom}, stating that for any $A,B\in \mathcal{B}$ and $f\in B-A$ there exists $e\in A-B$ such that $A-e+f\in\mathcal{B}$. Even more, both versions imply the existence of mutually exchangeable elements, that is, for any $A,B\in \mathcal{B}$ and $e\in A-B$ there exists $f\in B-A$ such that $A-e+f\in\cB$ and $B+e-f\in\mathcal{B}$; this property is referred to as \textbf{symmetric exchange property}. 


Given two bases $A$ and $B$ of a matroid $M$, the exchange axiom implies the existence of a sequence of exchanges that transforms $A$ into $B$, and the shortest length of such a sequence is $|A-B|$. However, the problem becomes significantly more difficult when it comes to the distance of basis pairs. Let $(A_1,A_2)$ be an ordered pair of bases of $M$, and assume that there exist $e\in A_1-A_2$ and $f\in A_2-A_1$ such that both $A'_1\coloneqq A_1-e+f$ and $A'_2\coloneqq A_2+e-f$ are bases. Then we say that the pair $(A'_1,A'_2)$ is obtained from $(A_1,A_2)$ by a \textbf{symmetric exchange}. The \textbf{exchange distance} of two basis pairs $(A_1,A_2)$ and $(B_1,B_2)$ is the minimum number of symmetric exchanges needed to transfer the former into the latter if such a sequence exists and $+\infty$ otherwise. A sequence of symmetric exchanges is called \textbf{strictly monotone for $(A_1,A_2)$ and $(B_1,B_2)$} if each step decreases the difference between the first members of the pairs. The basis pairs are omitted when those are clear from context.

As simple the case of single bases is, so little is known about the exchange distance of basis pairs. Can we characterize those cases when the exchange distance is finite? If the distance is finite, can we give an upper bound on it? Can we find a shortest exchange sequence algorithmically? Given basis pairs $(A_1,A_2)$ and $(B_1,B_2)$ of a rank-$r$ matroid, an obvious lower bound on their exchange distance is $r-|A_1\cap B_1|$. However, it might happen that more symmetric exchanges are needed even if $M$ is a graphic matroid; see \cite{farber1989basis} for a counterexample in $K_4$. 

Similar questions were considered before in different contexts. Gabow~\cite{gabow1976decomposing} studied the structure of symmetric exchanges and showed that for any pair $A,B\in\cB$, there exist partitions $A=X_1\cup\dots\cup X_q$, $B=Y_1\cup\dots\cup Y_q$ and bijections $\varphi^A_i\colon X_i\to Y_i$, $\varphi^B_i\colon Y_i\to X_i$ for $i=1,\dots,q$ such that 
\begin{eqnarray*}
&\left(\bigcup_{j=1}^{i-1}X_j\right)\cup \left((X_i-Z)\cup \varphi^A_i(Z)\right)\cup\left(\bigcup_{j=i+1}^q Y_{j}\right)\in\cB\quad\text{for $i=1,\dots,q$ and $Z\subseteq X_i$, and}&\\
&\left(\bigcup_{j=1}^{i-1}Y_j\right)\cup \left((Y_i-Z)\cup \varphi^B_i(Z)\right)\cup\left(\bigcup_{j=i+1}^q X_{j}\right)\in\cB\quad\text{for $i=1,\dots,q$ and $Z\subseteq Y_i$.}&
\end{eqnarray*}
In other words, the result shows the existence of a sequence of symmetric exchanges of subsets that can be executed element-by-element. In the light of this result, it is natural to ask whether $q$ can be made to take the extreme values, $q=1$ and $q=r$ where $r$ is the rank of the matroid. When $q=1$, then it is not difficult to see that $\varphi^A_1$ and $\varphi^B_1$ can be assumed to be inverses of each other, and the existence of such a bijection for any pair of bases is equivalent to the matroid being strongly base orderable. As strongly base orderable matroids form a proper subclass of matroids, $q=1$ cannot always be achieved. However, the other extreme case remained open, and Gabow~\cite{gabow1976decomposing} formulated the following beautiful problem, later stated as a conjecture by Wiedemann \cite{wiedemann1984cyclic} and by Cordovil and Moreira~\cite{cordovil1993bases}.

\begin{conj}[Gabow] \label{conj:gabow}
Let $A$ and $B$ bases of the same matroid. Then there are orderings $A=(a_1,\dots,a_r)$ and $B=(b_1,\dots,b_r)$ such that $\{a_1,\dots,a_i,b_{i+1},\dots,b_r\}$ and $\{b_1,\dots,b_i,a_{i+1},\dots,a_r\}$ are bases for $i=0,\dots,r$.
\end{conj}

We call this property the \textbf{sequential symmetric exchange property}. The conjecture is sometimes rephrased using cyclic orderings. Indeed, the statement is equivalent to the following: given two bases $A$ and $B$ of a rank-$r$ matroid, then the elements of $A$ and $B$ have a cyclic ordering such that both $A$ and $B$ form an interval, and any $r$ cyclically consecutive elements form a basis. 

A different extension of the exchange axiom was considered by White~\cite{white1980unique}, who proposed to characterize the equivalence of basis sequences. Let $\cX=(X_1,\dots,X_m)$ be a sequence of bases of a matroid $M$, and assume that there exist $e\in X_i-X_j$, $f\in X_j-X_i$ with $1\leq i<j\leq m$ such that both $X_i-e+f$ and $X_j+e-f$ are bases. The notion of symmetric exchanges can be naturally extended to sequences by saying that $\cX'=(X_1,\dots,X_{i-1},X_i-e+f,X_{i+1},\dots,X_{j-1},X_j+e-f,X_{j+1},\dots,X_m)$ is obtained from $\cX$ by a symmetric exchange. We call two sequences $\cX$ and $\cY$ \textbf{equivalent} if $\cY$ may be obtained from $\cX$ by a composition of symmetric exchanges. Note that for sequences of length two this is equivalent to the exchange distance of the pairs being finite. Furthermore, $\cX$ and $\cY$ are called \textbf{compatible} if $|\{i:\ s\in X_i,1\leq i\leq m\}|=|\{i:\ s\in Y_i,1\leq i\leq m\}|$ for every $s \in S$. Compatibility is obviously a necessary condition for two sequences being equivalent. White~\cite{white1980unique} conjectured that compatibility is also sufficient.

\begin{conj}[White] \label{conj:white}
Two basis sequences $\cX$ and $\cY$ of the same length are equivalent if and only if they are compatible.
\end{conj}

Despite being the center of attention not only in combinatorial optimization but also in algebra due to its connection to toric ideals and Gröbner bases~\cite{blasiak2008toric}, White's conjecture remains open even for sequences of length two. The close relation between the conjectures of White and Gabow is immediate: the latter would imply the former for sequences of the form $(A,B)$ and $(B,A)$. Indeed, a cyclic ordering of the elements defines a series of symmetric exchanges that transforms the sequence $(A,B)$ into the sequence $(B,A)$. 
White's conjecture is also closely related to the connectivity of basis pair graphs, and for basis sequences of length two, it is in fact equivalent to a conjecture of Farber~\cite{farber1989basis,farber1985edge}.



Equitability of matroids is a notion that is lesser known in the optimization community. A matroid $M=(S,\cB)$ is called \textbf{equitable} if for any set $X\subseteq S$ there exists a basis $B\in\cB$ such that $S-B$ is also a basis and $\lfloor|X|/2\rfloor\leq|B\cap X|\leq\lceil|X|/2\rceil$. Somewhat surprisingly, the following simple conjecture is open, see \cite{egres_open_equit}.

\begin{conj}[Equitability of matroids] \label{conj:equitable}
If the ground set of a matroid $M$ can be partitioned into two bases, then $M$ is equitable.
\end{conj}

A natural generalization of the problem would be asking for a partition into $k$ bases $B_1,\dots,B_k$ such that $\lfloor|X|/k\rfloor\leq|B_i\cap X|\leq\lceil|X|/k\rceil$. However, it is not difficult to show that the two variants are in fact equivalent, thus it suffices to concentrate on the first one, see \cite{egres_open_equit}. Not only the problem might be of independent combinatorial interest, but it is implied by both Gabow's and White's conjectures, therefore its validity is a precondition of the validity of the other two.

\paragraph{Previous work.}

It was already observed in \cite{gabow1976decomposing} that Conjecture~\ref{conj:gabow} holds for special classes of matroids, such as partition matroids, matching and transversal matroids, and matroids of rank less than $4$. An easy proof shows that it also holds for strongly base orderable matroids. The statement was later proved for graphic matroids by Wiedemann~\cite{wiedemann1984cyclic}, Kajitani, Ueno, and Miyano~\cite{kajitani1988ordering}, and Cordovil and Moreira~\cite{cordovil1993bases}. Kotlar and Ziv~\cite{kotlar2013serial} showed that any two elements of a basis have a sequential symmetric exchange with some two elements of any other basis. At the same time, Kotlar~\cite{kotlar2013circuits} proved that three consecutive symmetric exchanges exist for any two bases of a matroid, and that a full sequential symmetric exchange, of length at most $6$, exists for any two bases of a matroid of rank $5$.  Recently, Kotlar, Roda and Ziv~\cite{kotlar2021sequential} proposed a generalization that would extend the statement in the same way as the result of Greene and Magnanti~\cite{greene1975some} did for symmetric exchanges. Van den Heuvel and Thomass\'e~\cite{van2012cyclic} proposed a relaxation of the problem where the elements of the initial bases $A$ and $B$ do not have to form intervals. A natural extension would be to consider $k$ bases and to find a suitable cyclic ordering of the elements of these $k$ bases combined. This latter problem is wide open even for matroid classes for which the original conjecture is known to hold, including graphic matroids. Kajitani, Ueno and Miyano~\cite{kajitani1988ordering} observed that a necessary condition for the existence of a required cyclic ordering is the matroid $M$ being uniformly dense\footnote{A matroid is uniformly dense if $r\cdot |X|\leq r_M(X)\cdot |S|$ for every $X\subseteq S$, where $r_M$ and $r$ denote the rank function and the rank of $M$, respectively.}, and conjectured that this condition is also sufficient. A matroid admitting a partition of its ground set into bases is uniformly dense, therefore a proof for the conjecture of Kajitani, Ueno and Miyano would imply an affirmative answer to that of van den Heuvel and Thomass\'e as well. In \cite{van2012cyclic}, the conjecture was verified if $|S|$ and $r(S)$ are coprimes.

A routine argument shows that Conjecture~\ref{conj:white} holds for a matroid $M$ if and only if it holds for its dual $M^*$. For basis sequences of length two, Farber, Richter, and Shank~\cite{farber1985edge} verified the statement for graphic and cographic matroids, while Farber~\cite{farber1989basis} proved the conjecture for transversal matroids. For sequences of arbitrary length, the most significant partial result is due to Blasiak~\cite{blasiak2008toric}, who confirmed the conjecture for graphic matroids by showing that the toric ideal of a graphic matroid is generated by quadrics. Kashiwabara~\cite{kashiwabara2010toric} settled the case of matroids of rank at most $3$. Schweig~\cite{schweig2011toric} proved the statement for lattice path matroids, a subclass of transversal matroids. The case of strongly base orderable matroids was solved by Laso{\'n} and Micha{\l}ek~\cite{lason2014toric}. Recently, Blasiak's result was further extended by McGuinness~\cite{mcguinness2020frame} to frame matroids satisfying a linearity condition. Such classes of matroids include graphic matroids, bicircular matroids, signed graphic matroids, and more generally frame matroids obtained from group-labelled graphs. 


Apart from the matroid classes for which Conjecture~\ref{conj:gabow} or \ref{conj:white} was settled, Conjecture~\ref{conj:equitable} was shown to hold for base orderable matroids by  Fekete and Szabó~\cite{fekete2011equitable}. They also studied an extension of the problem where equitability with respect to more than one set is required. Király \cite{egres_open_equit} observed that the equitability of hypergraphic matroids follows from the graphic case. Equitable partitions are closely related to fair representations, introduced by Aharoni, Berger, Kotlar, and Ziv~\cite{aharoni2017faira}. For some positive real $\alpha$, a set $X$ represents another set $A$ $\alpha$-fairly if $|X\cap A|\geq\lfloor\alpha|A|\rfloor$, and $X$ represents $A$ almost $\alpha$-fairly if $|X\cap A|\geq\lfloor\alpha|A|\rfloor-1$. The notion of representation is then extended to partitions: a set is said to represent a partition $A_1,\dots,A_m$ $\alpha$-fairly (almost $\alpha$-fairly) if it represents all $A_i$'s $\alpha$-fairly (almost $\alpha$-fairly, respectively). The authors of~\cite{aharoni2017faira} provided several conjectures and results on fair representations. In particular, their results imply the following relaxation of Conjecture~\ref{conj:equitable}: If the ground set $S$ of a matroid $M$ can be partitioned into two bases, then for any $X\subseteq S$ there exists an independent set $I$ of size at least $r-1$ such that $S-I$ contains a basis, $|I\cap X|\geq\lfloor|X|/2\rfloor-1$, and $|I\cap (S-X)|\geq\lfloor|S-X|/2\rfloor-1$.

Our investigations were motivated by the results of Bonin~\cite{bonin2013basis} who verified Conjectures~\ref{conj:gabow}-\ref{conj:equitable} for sparse paving matroids. On a high level, our proofs resemble the proofs appearing in \cite{bonin2013basis}. Nevertheless, as one would expect, the higher complexity of split matroids compared to that of sparse paving matroids requires significantly more careful and intricate reasonings. 

\paragraph{Split matroids.}

While Conjectures~\ref{conj:gabow}-\ref{conj:equitable} were verified for various matroid classes, 
they remained open for paving matroids, a well-studied class with distinguished structural properties. 
Split matroids were introduced by Joswig and Schr\"oter~\cite{joswig2017matroids} via polyhedral geometry as a generalization of paving matroids, and quickly found applications in tropical geometry. Roughly speaking, a rank-$r$ matroid $M$ is split if the hyperplanes that are used to cut off parts of the matroid base polytope of the rank-$r$ uniform matroid to obtain that of $M$ satisfy a certain compatibility condition. The polyhedral point of view proved to be helpful in understanding the geometry of split matroids, and resulted in fundamental structural results on this class such as being closed both under duality and taking minors. Five forbidden minors were identified already in \cite{joswig2017matroids}, and Cameron and Mayhew~\cite{cameron2021excluded} later verified that the list is complete. 

Motivated by hypergraph representations of paving matroids, B\'erczi, Kir\'aly, Schwarcz, Yamaguchi and Yokoi~\cite{berczi2022hypergraph} initiated a combinatorial study of split matroids. They introduced the notion of \textbf{elementary split matroids}, a class that is a proper subclass of split matroids but includes all connected split matroids. The definition follows a combinatorial approach by setting the independent sets of the matroid to be the family of sets having bounded intersections with certain hyperedges. The proposed class captures all the nice properties of connected split matroids, and is closed not only under duality and taking minors but also truncation. More importantly, elementary split matroids can be characterized by a single forbidden minor. 

\paragraph{Our results.}

Our main contribution is a strong upper bound on the exchange distance of basis pairs in split matroids, together with a polynomial-time algorithm that determines a shortest exchange sequence\footnote{In matroid algorithms, it is usually assumed that the matroids are given by independence oracles. Then the complexity of an algorithm is measured by the number of oracle calls and other conventional elementary steps.}. Given compatible basis pairs $(A_1,A_2)$ and $(B_1,B_2)$ of a rank-$r$ split matroid, we prove that $\min\{r,r-|A_1\cap B_1|+1\}$ such steps always suffice. This result immediately implies that, for split matroids, Conjectures~\ref{conj:gabow} and \ref{conj:equitable} hold in general, and Conjecture~\ref{conj:white} hold for sequences of length two. The graphic matroid of $K_4$ is sparse paving and so it is a split matroid, hence the example of \cite{farber1989basis} shows that the bound above is best possible.

As a corollary, we prove that for split matroids there always exists a strictly monotone sequence of symmetric exchanges of length $r-3|A_1\cap B_1|$. For base orderable split matroids and paving matroids, we improve the bound to $r-2|A_1\cap B_1|$ and $r-|A_1\cap B_1|-2$, respectively. Note that $r-|A_1\cap B_1|$ is clearly an upper bound on the maximum length of a strictly monotone sequence. 

\medskip

The rest of the paper is organized as follows. Basic notation and definitions are given in Section~\ref{sec:pre}. Section~\ref{sec:main} provides a polynomial-time algorithm that determines the exchange distance of basis pairs in split matroids, and discusses Conjectures~\ref{conj:gabow}-\ref{conj:equitable} for this class. Section~\ref{sec:monotone} provides lower bounds on the maximum length of a strictly monotone sequence of symmetric exchanges in various settings.

\section{Preliminaries}
\label{sec:pre}

Throughout, we denote by $S$ the ground set of a matroid. For subsets $X,Y\subseteq S$, 
the \textbf{difference} of $X$ and $Y$ is denoted by $X-Y$. If $Y$ consists of a single element $y$, then $X-\{y\}$ and $X\cup \{y\}$ are abbreviated as $X-y$ and $X+y$, respectively. 
An ordered sequence consisting of elements of $S$ is denoted by $\ul{x}=(x_1,\dots,x_s)$.

For basic definitions on matroids, we refer the reader to~\cite{oxley2011matroid}. 
Let $M=(S,\cB)$ be a rank-$r$ matroid on $S$, where $\cB$ is the family of bases of $M$. We denote the \textbf{rank function} of $M$ by $r_M$. A matroid of rank $r$ is called \textbf{paving} if every set of size at most $r-1$ is independent, or in other words, every circuit of the matroid has size at least $r$. A matroid is \textbf{sparse paving} if it is both paving and dual to a paving matroid. For a non-negative integer $r$, a set $S$ of size at least $r$, and a (possibly empty) family $\mathcal{H}=\{H_1,\dots,H_q\}$ of proper subsets of $S$ such that $|H_i\cap H_j|\leq r-2$ for $1\leq i < j\leq q$, the set system $\mathcal{I}=\{X\subseteq S\mid |X|\leq r,\ |X\cap H_i|\leq r-1\ \text{for } i=1,\dots,q\}$ forms the family of independent sets of a rank-$r$ paving matroid, and in fact every paving matroid can be obtained in this form, see \cite{hartmanis1959lattice,welsh1976matroid}. 

Joswig and Schr\"oter originally defined split matroids via polyhedral geometry. They showed that split matroids are closed under taking minors~\cite[Proposition 44]{joswig2017matroids}, and that it suffices to concentrate on the connected case~\cite[Proposition 15]{joswig2017matroids}.

\begin{prop}[Joswig and Schr\"oter]\label{prop:conn}
\mbox{}

\begin{enumerate}[label={(\alph*)}] \itemsep0em
    \item The class of split matroids is minor-closed. \label{it:min}
    \item  A matroid is a split matroid if and only if at most one connected component is a non-uniform split matroid and all other components are uniform. \label{it:conn}
\end{enumerate}
\end{prop}

In \cite{berczi2022hypergraph}, a proper subclass called \textbf{elementary split matroids} was introduced which includes all connected split matroids. The main advantage of the proposed subclass is that it admits a hypergraph characterization similar to that of paving matroids. Namely, let $S$ be a ground set of size at least $r$, $\cH=\{H_1,\dots, H_q\}$ be a (possibly empty) collection of subsets of $S$, and $r, r_1, \dots, r_q$ be non-negative integers satisfying
\begin{align}
|H_i \cap H_j| &\le r_i + r_j -r\ \text{for $1 \le i < j \le q$,}\tag*{(H1)}\label{eq:h1}\\
|S-H_i| + r_i &\ge r\ \text{for $i=1,\dots, q$.} \tag*{(H2)}\label{eq:h2}
\end{align}
Then $\cI=\{X\subseteq S\mid |X|\leq r,\ |X\cap H_i|\leq r_i\ \text{for $1\leq i \leq q$}\}$ forms the family of independent sets of a rank-$r$ matroid $M$ with rank function $r_M(Z)=\min\big\{r,|Z|,\min_{1\leq i\leq q}\{|Z-H_i|+r_i\}\big\}$. Matroids that can be obtained in this form are called \textbf{elementary split matroids}. It was observed already in \cite{berczi2022hypergraph} that the underlying hypergraph can be chosen in such a way that 
\begin{align}
r_i &\le r-1\  \text{for $i=1,\dots, q$,} \tag*{(H3)}\label{eq:h3}\\
|H_i| &\ge r_i+1\  \text{for $i=1,\dots, q$.} \tag*{(H4)}\label{eq:h4}
\end{align}
Therefore, we call the representation \textbf{non-redundant} if all of \ref{eq:h1}--\ref{eq:h4} hold. A set $F\subseteq S$ is called \textbf{$H_i$-tight} if $|F\cap H_i|=r_i$.
Finally, we will use the following result from~\cite[Theorem 11]{berczi2022hypergraph} that establishes a connection between elementary and connected split matroids.

\begin{prop}[B\'erczi, Kir\'aly, Schwarcz, Yamaguchi and Yokoi]\label{prop:esplit}
Connected split matroids form a subclass of elementary split matroids.
\end{prop}

\section{Exchange distance of basis pairs}
\label{sec:main}

Let $(A_1,A_2)$ and $(B_1,B_2)$ be compatible basis pairs of a rank-$r$ split matroid $M$, that is, $A_1\cap A_2=B_1\cap B_2$ and $A_1\cup A_2=B_1\cup B_2$. The goal of the current section is to give an upper bound on the number of symmetric exchanges needed to transform $(A_1,A_2)$ into $(B_1,B_2)$. The high-level idea of the proof is as follows. We recursively try to pair up the elements of $A_1\cap B_2$ and $A_2\cap B_1$ in such a way that each pair $x_i,y_i$ corresponds to a symmetric exchange between bases $A_1-\{x_1,\dots,x_{i-1}\}\cup\{y_1,\dots,y_{i-1}\}$ and $A_2-\{y_1,\dots,y_{i-1}\}\cup\{x_1,\dots,x_{i-1}\}$. If no such pair exists, then we identify the obstacle blocking every possible exchange in terms of hyperedges appearing in a non-redundant representation of $M$. At this point we show that there exists a symmetric exchange of elements from $A_1\cap B_1$ and $A_2\cap B_1$ that can be followed by a complete sequence of symmetric exchanges between elements of $A_1\cap B_2$ and $A_2\cap B_1$, leading to the $r-|A_1\cap B_1|+1$ bound. 

In our proof, we will rely on the following simple technical lemma several times. 

\begin{lem}\label{lem:tight}
Let $M$ be a rank-$r$ elementary split matroid with a non-redundant representation $\cH=\{H_1,\dots,H_q\}$ and $r,r_1,\dots,r_q$. Let $F$ be a set of size $r$.
\begin{enumerate}[label={(\alph*)}] \itemsep0em
    \item If $F$ is $H_i$-tight for some index $i$ then $F$ is a basis of $M$.\label{it:basis}
    \item If $F$ is both $H_i$-tight and $H_j$-tight for distinct indices $i$ and $j$ then $H_i\cap H_j\subseteq F\subseteq H_i\cup H_j$.\label{it:f}
\end{enumerate} 
\end{lem}
\begin{proof}
Assume that $F$ is $H_i$-tight. Then, by \ref{eq:h1}, for any index $j\neq i$ we get
\begin{align*}
    |F\cap H_j|
    {}&{}= 
    |F\cap (H_i\cup H_j)|+|F\cap (H_i\cap H_j)|-|F\cap H_i|\\
    {}&{}\leq 
    |F|+|H_i\cap H_j|-r_i\\
    {}&{}\leq 
    r+(r_i+r_j-r)-r_i\\
    {}&{}=
    r_j,
\end{align*}
showing the $F$ is indeed a basis. If $F$ is also $H_j$-tight, then equality holds throughout, implying $H_i\cap H_j\subseteq F\subseteq H_i\cup H_j$. 
\end{proof}

To avoid referring to part \ref{it:basis} of the lemma too often, we implicitly consider sets of size $r$ that are tight with respect to one of the hyperedges to be bases. Now we turn to the proof of the main result of the paper.

\begin{thm}\label{thm:main}
Let $(A_1,A_2)$ and $(B_1,B_2)$ be compatible basis pairs of a rank-$r$ split matroid $M$. Then the exchange distance of $(A_1,A_2)$ and $(B_1,B_2)$ is at most $\min\{r,r-|A_1\cap B_1|+1\}$.
\end{thm}
\begin{proof}
Recall that $(A_1,A_2)$ and $(B_1,B_2)$ are compatible if $A_1\cap A_2=B_1\cap B_2$ and $A_1\cup A_2=B_1\cup B_2$. We start with two simplifications. First, we may assume that $A_1\cap A_2=\emptyset$. Indeed, if $A_1\cap A_2\neq\emptyset$, then let $A_1'\coloneqq A_1-A_2$, $A_2'\coloneqq A_2-A_1$, $B_1'\coloneqq B_1-B_2$, $B_2'\coloneqq B_2-B_1$ and let $M'$ denote the matroid obtained from $M$ by contracting $A_1\cap A_2$. By Proposition~\ref{prop:conn}\ref{it:min}, $M'$ is also a split matroid in which both $A_1',A_2'$ and $B_1',B_2'$ are disjoint bases. Note that the rank of $M'$ is $r'\coloneqq r-|A_1\cap A_2|$. As a sequence of symmetric exchanges that transforms $(A_1',A_2')$ into $(B_1',B_2')$ also transforms $(A_1,A_2)$ into $(B_1,B_2)$, the claim follows by $\min\{r',r'-|A_1'\cap B_1'|+1\}\le \min\{r,r-|A_1\cap B_1|+1\}$.

Second, it suffices to verify the statement for connected split matroids. Indeed, if $M$ is not connected, then all but at most one of its components are uniform matroids by Proposition~\ref{prop:conn}\ref{it:conn}. Let $M_0=(S_0,\cB_0)$ denote, if exists, the non-uniform component, and let $M_1=(S_1,\cB_1),\dots,M_t=(S_t,\cB_t)$ be the uniform connected components of $M$. Define $A_1^i\coloneqq A_1\cap S_i$, $A_2^i\coloneqq A_2\cap S_i$, $B_1^i\coloneqq B_1\cap S_i$, $B_2^i\coloneqq B_2\cap S_i$ for $i=0,\dots,t$. Note that $A_1^i$, $A_2^i$, $B_1^i$ and $B_2^i$ are bases of $M_i$. Take a sequence of at most $\min\{r_0,r_0-|A_1^0\cap B_1^0|+1\}$ symmetric exchanges that transforms $(A^0_1,A^0_2)$ into $(B^0_1,B^0_2)$, where $r_0$ is the rank of $M_0$. As $M_i$ is a uniform matroid, it is easy to see that the exchange distance of $(A_1^i,A_2^i)$ and $(B_1^i,B_2^i)$ is exactly $r_i-|A^i_1\cap B^i_1|$; take such a sequence  for $1\leq i\leq t$. Then the concatenations of these sequences result in a sequence of at most $\min\{r_0,r_0-|A_1^0\cap B_1^0|+1\}+\sum_{i=1}^t [r_i-|A_1^i\cap B_1^i|]\leq \min\{r,r-|A_1\cap B_1|+1\}$ symmetric exchanges that transforms $(A_1,A_2)$ into $(B_1,B_2)$. 

Hence we may assume that $M$ is connected. Let $\cH$ be a hypergraph corresponding to a non-redundant representation of $M$; such a representation exists by Proposition~\ref{prop:esplit}. By convention, we denote the bound on a hyperedge $H_i$ by $r_i$. Choose a pair of inclusionwise maximal sequences $\ul{x}\coloneqq(x_1,\dots,x_s)$ and $\ul{y}\coloneqq(y_1,\dots,y_s)$ of distinct elements greedily that correspond to a strictly monotone sequence of exchanges, that is,
\begin{align*}
  \tag*{($\star$)}\label{star}
  \renewcommand{\arraystretch}{1.2}
  \begin{array}{c}
x_i\in A_1\cap B_2,\ y_i\in A_2\cap B_1,\\ 
(A_1-\{x_1,\dots,x_i\})\cup\{y_1,\dots,y_i\}\ \text{is a basis for $i=1,\dots,s$,}\\
(A_2-\{y_1,\dots,y_i\})\cup\{x_1,\dots,x_i\}\ \text{is a basis for $i=1,\dots,s$.}
  \end{array}
\end{align*}
Define $A_1'\coloneqq (A_1-\{x_1,\dots,x_s\})\cup\{y_1,\dots,y_s\}$ and $A_2'\coloneqq (A_2-\{y_1,\dots,y_s\})\cup \{x_1,\dots,x_s\}$. If $A_1'=B_1$, then these sequences correspond to symmetric exchanges transforming $(A_1,A_2)$ into $(B_1,B_2)$, and we are done. Otherwise $A_1'\neq B_1$ and, by the maximal choice of the sequences, there is no $x\in A_1'\cap B_2$ and $y\in A_2'\cap B_1$ such that both $A_1'-x+y$ and $A_2'+x-y$ are bases.

\begin{cl}\label{cl:h}
There exist distinct hyperedges $H_1$, $H_2$, $H_3$ and $H_4$ such that $A'_1$ is $H_1$- and $H_3$-tight, while $A'_2$ is $H_2$- and $H_4$-tight.
\end{cl}
\begin{claimproof}
Fix an arbitrary $x\in A_1'\cap B_2$ and let $y\in A_2'\cap B_1$ be an element for which $A_2'+x-y$ is a basis. Note that such a $y$ exists by the co-exchange axiom for $A_2'$, $B_2$ and $x\in B_2-A_2'$. By our assumption, $A_1'-x+y$ is not a basis, hence there exists a hyperedge $H_1$ with $|(A_1'-x+y)\cap H_1|>r_1$. As $A_1'$ is a basis, we get $|A_1'\cap H_1|=r_1$, $x\notin H_1$ and $y\in H_1$. 

By applying the co-exchange axiom for $A_1'$, $B_1$ and $y\in B_1-A_1'$, there exists an $x'\in A_1'\cap B_2$ for which $A_1'-x'+y$ is a basis. By our assumption, $A_2'+x'-y$ is not a basis, hence there exists a hyperedge $H_2$ with $|(A_2'+x'-y)\cap H_2|>r_2$. As $A_2'$ is a basis, we get $|A_2'\cap H_2|=r_2$, $y\notin H_2$ and $x'\in H_2$. Necessarily, $x'\in H_1$ as otherwise $|(A_1'-x'+y)\cap H_1|=r_1+1$, contradicting $A_1'-x'+y$ being a basis. Similarly, $x\notin H_2$ as otherwise $|(A_2'+x-y)\cap H_2|=r_2+1$, contradicting $A_2'+x-y$ being a basis.

By applying the co-exchange axiom for $A_2'$, $B_2$ and $x'\in B_2-A_2'$, there exists a $y'\in A_2'\cap B_1$ for which $A_2'+x'-y'$ is a basis. By our assumption, $A_1'-x'+y'$ is not a basis, hence there exists a hyperedge $H_3$ with $|(A_1'-x'+y')\cap H_3|>r_3$. As $A_1'$ is a basis, we get $|A_1'\cap H_3|=r_3$, $x'\notin H_3$ and $y'\in H_3$. The set $A_1'$ is both $H_1$- and $H_3$-tight, hence $H_1\cap H_3\subseteq A_1'\subseteq H_1\cup H_3$ by Lemma~\ref{lem:tight}\ref{it:f}. As $x\in A_1'-H_1$, $y\in H_1-A_1'$ and $y'\in H_3-A_1'$, we get $x\in H_3$, $y\notin H_3$ and $y'\notin H_1$, respectively.

Being $H_1$-tight, the set $A_1'-x+y'$ is a basis. Therefore $A_2'+x-y'$ is not a basis, implying the existence of a hyperedge $H_4$ with $|(A_2'+x-y')\cap H_4|>r_4$. As $A_2'$ is a basis, we get $|A_2'\cap H_4|=r_4$, $y'\notin H_4$ and $x\in H_4$. The set $A_2'$ is both $H_2$- and $H_4$-tight, hence $H_2\cap H_4\subseteq A_2'\subseteq H_2\cup H_4$ by Lemma~\ref{lem:tight}\ref{it:f}. As $y\in A_2'-H_2$, $x'\in H_2-A_2'$ and $y'\in A_2'-H_4$, we get $y\in H_4$, $x'\notin H_4$ and $y'\in H_2$.

Note that the pairs $\{x',y\}$, $\{x',y'\}$, $\{x,y'\}$ and $\{x,y\}$ are included only in $H_1$, $H_2$, $H_3$ and $H_4$, respectively, hence these hyperedges are pairwise distinct.
\end{claimproof}

Our goal is to recover the relations of the four hyperedges identified by Claim~\ref{cl:h}. For ease of discussion, we denote by $H_{i,j}\coloneqq(H_i\cap H_j)-(H_k\cup H_\ell)$ for $\{i,j,k,\ell\}=\{1,2,3,4\}$. The proof of the claim also implies that none of the sets $A'_1\cap B_2\cap H_{3,4}$, $A'_1\cap B_2\cap H_{1,2}$, $A'_2\cap B_1\cap H_{1,4}$ and $A'_2\cap B_1\cap H_{2,3}$ is empty, see the elements $x$, $x'$, $y$ and $y'$.

\begin{cl}\label{cl:symdiff}
$A'_1\cap B_2\subseteq H_{1,2}\cup H_{3,4}$ and 
$A'_2\cap B_1\subseteq H_{1,4}\cup H_{2,3}$.
\end{cl}
\begin{claimproof}
Take an arbitrary $z\in A'_1\cap B_2$. As $z\in A'_1-A'_2$, Claim~\ref{cl:h} and Lemma~\ref{lem:tight}\ref{it:f} together imply $z\in H_1\cup H_3$ and $z\notin H_2\cap H_4$. Furthermore, $z\notin H_1-H_2$ as otherwise $A'_1-z+y$ is $H_1$-tight and $A'_2+z-y$ is $H_2$-tight for any $y\in A'_2\cap B_1\cap H_{1,4}$, contradicting the assumption that no such symmetric exchange exists. Similarly, $z\notin H_3-H_4$ as otherwise $A'_1-z+y'$ is $H_3$-tight and $A'_2-y'+z$ is $H_4$-tight for any $y'\in A'_2\cap B_1\cap H_{2,3}$, contradicting the assumption that no such symmetric exchange exists. These observations imply $A'_1\cap B_2\subseteq H_{1,2}\cup H_{3,4}$.

The inclusion $A'_2\cap B_1\subseteq H_{1,4}\cup H_{2,3}$ can be proved analogously. 
\end{claimproof}

Next we show that either the length of the strictly monotone sequence of exchanges can be increased, or we derive further structural observations on the hyperedges $H_1$, $H_2$, $H_3$ and $H_4$.

\begin{cl}\label{cl:schwarcz}
Either there exists a pair of sequences $\ul{x}'=(x'_1,\dots,x'_{s+2})$ and $\ul{y}'=(y'_1,\dots,y'_{s+2})$ satisfying \ref{star}, or $\{x_i,y_i\}\subseteq (H_1\triangle H_3)\cap (H_2\triangle H_4)$ and $x_i$ and $y_i$ are contained in the same hyperedges from $\{H_1,H_2,H_3,H_4\}$ for each $1 \le i \le s$. 
\end{cl}
\begin{claimproof}
We prove the claim by considering the pairs $x_i,y_i$ starting from $i=s$ in a reversed order. In a general phase, consider an index $1\leq i\leq s$ and assume that $x_j$ and $y_j$ are contained in the same hyperedges from $\{H_1, H_2, H_3, H_4\}$ for each $i+1 \le j \le s$. This assumption and Claim~\ref{cl:h} imply that 
\begin{gather*}
    (A_1-\{x_1,\dots,x_j\})+\{y_1,\dots,y_j\}\ \text{is $H_1$- and $H_3$-tight for $i+1\leq j\leq s$},\\ 
    (A_2-\{y_1,\dots,y_j\})+\{x_1,\dots,x_j\}\ \text{is $H_2$- and $H_4$-tight for $i+1\leq j\leq s$}.
\end{gather*} 
By Claim~\ref{cl:h} and Lemma~\ref{lem:tight}\ref{it:f}, we have $x_i \not \in H_1 \cap H_3$, $x_i \in H_2 \cup H_4$, $y_i \in H_1 \cup H_3$ and $y_i \not \in H_2\cap H_4$. We prove that either we can construct a pair $\ul{x}',\ul{y}'$ of sequences of length $s+2$ satisfying the conditions of \ref{star}, or $\{x_i, y_i\}\subseteq (H_1 \triangle H_3) \cap (H_2 \triangle H_4)$ and $x_i$ and $y_i$ are contained in the same hyperedges from $\{H_1,H_2,H_3,H_4\}$. Fix elements $x\in A'_1\cap B_2\cap H_{3,4}$, $x'\in A'_1\cap B_2\cap H_{1,2}$, $y\in A'_2\cap B_1\cap H_{1,4}$ and $y'\in A'_2\cap B_1\cap H_{2,3}$. We discuss four cases.
\smallskip

\noindent\textbf{Case 1.} $x_i \not \in H_1 \cup H_3$.

\noindent If $x_i \in H_2$, define $\ul{x}'\coloneqq(x_1,\dots,x_{i-1},x',x_{i+1},\dots,x_s,x,x_i)$ and $\ul{y}'\coloneqq(y_1,\dots,y_s,y,y')$. Then 
\begin{gather*}
    (A_1-\{x'_1,\dots,x'_j\})+\{y'_1,\dots,y'_j\}\ \text{is $H_3$-tight for $i \le j \le s$ and $H_1$-tight for $j=s+1,s+2$},\\
    (A_2-\{y'_1,\dots,y'_j\})+\{x'_1,\dots,x'_j\}\ \text{is $H_2$-tight for $i \le j \le s+2$}.
\end{gather*}  
If $x_i\notin H_2$, then $x_i\in H_4$. Consider the sequences $\ul{x}'\coloneqq(x_1,\dots,x_{i-1},x,x_{i+1},\dots,x_s,x',x_i)$ and $\ul{y}'\coloneqq(y_1,\dots,y_s,y',y)$. Then 
\begin{gather*}
    (A_1-\{x'_1,\dots,x'_j\})+\{y'_1,\dots,y'_j\}\ \text{is $H_1$-tight for $i \le j \le s$ and $H_3$ tight for $j=s+1,s+2$},\\
    (A_2-\{y'_1,\dots,y'_j\})+\{x'_1,\dots,x'_j\}\ \text{is $H_2$-tight for $i \le j \le s+2$}.
\end{gather*}  

\noindent\textbf{Case 2.} $x_i \in H_2 \cap H_4$.

\noindent If $x_i \notin H_1$, define $\ul{x}'\coloneqq (x_1,\dots,x_{i-1},x,x_{i+1},\dots,x_s,x',x_i)$ and $\ul{y}'\coloneqq(y_1,\dots,y_s,y,y')$. Then 
\begin{gather*}
    (A_1-\{x'_1,\dots,x'_j\})+\{y'_1,\dots,y'_j\}\ \text{is $H_1$-tight for $i \le j \le s+2$},\\
    (A_2-\{y'_1,\dots,y'_j\})+\{x'_1,\dots,x'_j\}\ \text{is $H_4$-tight for $i \le j \le s+1$ and $H_2$-tight for $j=s+1,s+2$}.\\
\end{gather*}  
If $x_i\in H_1$, then $x_i\notin H_3$. Consider the sequences $\ul{x}'\coloneqq(x_1,\dots,x_{i-1},x',x_{i+1},\dots,x_s,x,x_i)$ and $\ul{y}'\coloneqq(y_1,\dots,y_s,y',y)$. Then
\begin{gather*}
    (A_1-\{x'_1,\dots,x'_j\})+\{y'_1,\dots,y'_j\}\ \text{is $H_1$-tight for $i \le j \le s+2$},\\
    (A_2-\{y'_1,\dots,y'_j\})+\{x'_1,\dots,x'_j\}\ \text{is $H_2$-tight for $i \le j \le s$ and $H_4$-tight for $j=s+1,s+2$}.
\end{gather*}

\noindent\textbf{Case 3.} $y_i \in H_1 \cap H_3$.

\noindent If $y_i \notin H_2$, define $\ul{x}'\coloneqq(x_1,\dots,x_s,x',x)$ and $\ul{y}'\coloneqq(y_1,\dots,y_{i-1},y,y_{i+1},\dots,y_s,y',y_i)$. Then 
\begin{gather*}
    (A_1-\{x'_1,\dots,x'_j\})+\{y'_1,\dots,y'_j\}\ \text{is $H_1$-tight for $i \le j \le s$ and $H_3$-tight for $j=s+1,s+2$},\\
    (A_2-\{y'_1,\dots,y'_j\})+\{x'_1,\dots,x'_j\}\ \text{is $H_2$-tight for $i \le j \le s+2$}.
\end{gather*}  
If $y_i\in H_2$, then $y_i\notin H_4$. Consider the sequences $\ul{x}'\coloneqq(x_1,\dots,x_s,x,x')$ and $\ul{y}'\coloneqq(y_1,\dots,y_{i-1},\allowbreak y',y_{i+1},\dots,y_s,y,y_i)$. Then 
\begin{gather*}
    (A_1-\{x'_1,\dots,x'_j\})+\{y'_1,\dots,y'_j\}\ \text{is $H_3$-tight for $i \le j \le s$ and $H_1$-tight for $j=s+1,s+2$},\\
    (A_2-\{y'_1,\dots,y'_j\})+\{x'_1,\dots,x'_j\}\ \text{is $H_4$-tight for $i \le j \le s$}.
\end{gather*}  

\noindent\textbf{Case 4.} $y_i \notin H_2 \cup H_4$.

\noindent If $y_i \in H_1$, define $\ul{x}'\coloneqq(x_1,\dots,x_s,x,x')$ and $\ul{y}'\coloneqq(y_1,\dots,y_{i-1},y,y_{i+1},\dots,y_s,y',y_i)$. Then 
\begin{gather*}
    (A_1-\{x'_1,\dots,x'_j\})+\{y'_1,\dots,y'_j\}\ \text{is $H_1$-tight for $i \le j \le s+2$},\\
    (A_2-\{y'_1,\dots,y'_j\})+\{x'_1,\dots,x'_j\}\ \text{is $H_2$-tight for $i \le j \le s$ and $H_4$-tight for $j=s+1,s+2$}.
\end{gather*}  
If $y_i\notin H_1$, then $y_i\in H_3$. Consider the sequences $\ul{x}'\coloneqq(x_1,\dots,x_s,x',x)$ and $\ul{y}'\coloneqq(y_1,\dots,y_{i-1},\allowbreak y',y_{i+1},\dots,y_s,y,y_i)$. Then 
\begin{gather*}
    (A_1-\{x'_1,\dots,x'_j\})+\{y'_1,\dots,y'_j\}\ \text{is $H_3$-tight for $i \le j \le s+2$},\\
    (A_2-\{y'_1,\dots,y'_j\})+\{x'_1,\dots,x'_j\}\ \text{is $H_4$-tight for $i \le j \le s$ and $H_2$-tight for $j=s+1,s+2$}.
\end{gather*}  

By the above, every basis that is affected by the changes in $\ul{x}$ and $\ul{y}$ remains tight, and hence a basis. In addition, $x'_{s+1},y'_{s+1}$ and $x'_{s+2},y'_{s+2}$ are also symmetric exchanges for the corresponding bases, therefore $\ul{x}'$ and $\ul{y}'$ satisfy the conditions of \ref{star}.

If none of Cases 1-4 applies, then $\{x_i,y_i\}\subseteq (H_1\triangle H_3)\cap (H_2\triangle H_4)$. Therefore there exist unique indices $k, k' \in \{1,3\}$ and $\ell, \ell' \in \{2,4\}$ such that $x_i \in H_k \cap H_\ell$ and $y_i \in H_{k'} \cap H_{\ell'}$. As $(A_1-\{x_1,\dots,x_i\})\cup\{y_1,\dots,y_i\}$ is $H_k$-tight and $(A_1-\{x_1,\dots,x_{i-1}\})\cup\{y_1,\dots,y_{i-1}\}$ is a basis, it follows that $y_i \in H_k$, that is, $k=k'$. As $(A_2-\{y_1,\dots,y_i\})\cup\{x_1,\dots,x_i\}$ is $H_{\ell'}$-tight and $(A_2-\{y_1,\dots,y_{i-1}\})\cup\{x_1,\dots,x_{i-1}\}$ is a basis, it follows that $x_i \in H_{\ell '}$, that is, $\ell = \ell'$. This concludes the proof of the claim. 
\end{claimproof}

If possible, we increase the length of $\ul{x}$ and $\ul{y}$ using Claim~\ref{cl:schwarcz}, and start again the whole procedure. Hence assume that this is not the case, that is, $A'_1\neq B_1$ and $\{x_i,y_i\}\subseteq (H_1\triangle H_3)\cap (H_2\triangle H_4)$ for each $1 \le i \le s$. Let us denote by $d\coloneqq |A'_1-B_1|=|A'_2-B_2|$. 

\begin{cl}\label{cl:tight}
$|(A'_k-B_k)\cap H_i|=d/2$ for $k=1,2$ and $1\leq i\leq 4$. Furthermore, $A_1$ and $B_1$ are $H_1$- and $H_3$-tight, while $A_2$ and $B_2$ are $H_2$- and $H_4$-tight.
\end{cl}
\begin{claimproof} 
We know that $A'_1$ is $H_1$- and $H_3$-tight, and $A'_2$-is $H_2$- and $H_4$-tight. As $x_i$ and $y_i$ are contained in the same hyperedges from $\{H_1, H_2, H_3, H_4\}$ for each $1\le i \le s$ by Claim~\ref{cl:schwarcz}, these immediately imply that $A_1$ is $H_1$- and $H_3$-tight, and $A_2$ is $H_2$- and $H_4$-tight. 
As $A'_1$ is $H_k$-tight for $k\in\{1,3\}$ and $B_1$ is a basis, we get that \[|A'_1 \cap H_k| = r_k \ge |B_1 \cap H_k| = |A'_1 \cap H_k| - |A'_1 \cap B_2 \cap H_k| + |A'_2 \cap B_1 \cap H_k|,\]
hence 
\[|A'_1 \cap B_2 \cap H_k| \ge |A'_2 \cap B_1 \cap H_k| \text{ for $k\in \{1,3\}$}.\]
Similarly, as $B_2$ is a basis and $A'_2$ is $H_k$-tight for $k\in\{2,4\}$, we have
\[|A'_2 \cap B_1 \cap H_k| \ge |A'_1 \cap B_2 \cap H_k| \text{ for $k\in \{2,4\}$}.\]
Using these inequalities and Claim~\ref{cl:symdiff}, it follows that
\begin{align*}
    |A'_1\cap B_2\cap H_1|
    {}&{}\geq |A'_2\cap B_1\cap H_1|=|A'_2\cap B_1\cap H_4|\\
    {}&{}\geq |A'_1\cap B_2\cap H_4|=|A'_1\cap B_2\cap H_3|\\
    {}&{}\geq |A'_2\cap B_1\cap H_3|=|A'_2\cap B_1\cap H_2|\\
    {}&{}\geq |A'_1\cap B_2\cap H_2|=|A'_1\cap B_2\cap H_1|.
\end{align*}
Therefore equality holds throughout. This implies that the common size of these sets is $d/2$, and that $B_1$ is $H_1$- and $H_3$-tight, while $B_2$ is $H_2$- and $H_4$-tight.
\end{claimproof}

Note that $d$ can be interpreted as the number of elements that could not be included in the strictly monotone sequence of exchanges. The next claim shows that $d$ is not too large compared to $|A_1\cap B_1|$.

\begin{cl}\label{cl:d}
$d\leq 2|A_1\cap B_1|$.
\end{cl}
\begin{claimproof}
By combining Claims~\ref{cl:symdiff} and \ref{cl:tight}, we get that $|A'_1 \cap B_2 \cap H_{1,2}| = |A'_1 \cap B_2 \cap H_{3,4}| = d/2$, thus $|(A_1 \cap B_2) - (H_1 \cup H_2)| \ge |(A'_1 \cap B_2) - (H_1 \cup H_2) | = d/2$. Using that $A_1$ is $H_1$-tight and $B_2$ is $H_2$-tight by Claim~\ref{cl:tight}, we get
\begin{align*}
    |H_1\cap H_2|
    {}&{}\leq
    r_1+r_2-r\\
    {}&{} = |A_1\cap H_1|+|B_2\cap H_2|-r\\
    {}&{} = |A_1\cap B_1\cap H_1|+|A_1\cap B_2\cap H_1|+|A_1\cap B_2\cap H_2|+|A_2\cap B_2\cap H_2|-r\\
    {}&{} \le |A_1 \cap B_1| + |A_1 \cap B_2 \cap H_1| + |A_1 \cap B_2 \cap H_2| + |A_2 \cap B_2| -r \\ 
    {}&{} = |A_1 \cap B_1| + |A_1 \cap B_2 \cap H_1| + |A_1 \cap B_2 \cap H_2| - |A_1 \cap B_2| \\ 
    {}&{} = |A_1\cap B_1|+|(A_1 \cap B_2) \cap (H_1 \cap H_2)| - |(A_1\cap B_2)-(H_1\cup H_2)|\\
    {}&{}\leq
    |A_1\cap B_1|+|H_1\cap H_2|-d/2 
\end{align*}
concluding the proof.
\end{claimproof}

As we assumed that $d>0$ holds, the claim implies $A_1\cap B_1\neq \emptyset$, meaning that $\min\{r,r-|A_1\cap B_1|+1\}=r-|A_1\cap B_1|+1$. Hence it suffices to show that the exchange distance of $(A'_1,A'_2)$ and $(B_1,B_2)$ is at most $d+1$. We will directly construct the remaining exchanges by relying on the existence of an element in $A'_1\triangle B_2=A'_2\triangle B_1$ that is contained in exactly one or three of the hyperedges $H_1$, $H_2$, $H_3$ and $H_4$.

\begin{cl}\label{cl:z}
$(A'_1\cap B_1)\cup(A'_2\cap B_2)\not\subseteq H_1\triangle H_3$ and $(A'_1\cap B_1)\cup(A'_2\cap B_2)\not\subseteq H_2\triangle H_4$.
\end{cl}
\begin{claimproof} 
We prove $(A'_1 \cap B_1) \cup (A'_2 \cap B_2) \not\subseteq H_2 \triangle H_4$, the inclusion $(A'_1 \cap B_1) \cup (A'_2 \cap B_2) \not\subseteq H_1 \triangle H_3$ can be proved analogously. Suppose indirectly that $(A'_1\cap B_1)\cup(A'_2\cap B_2)\subseteq H_2\triangle H_4$. Then, by Claim~\ref{cl:symdiff}, we get
\begin{align*}
r=|A'_1|=|A'_1\cap H_2|+|A'_1\cap H_4|
\leq r_2+r_4=|A'_2\cap H_2|+|A'_2\cap H_4|
= |A'_2|=r.
\end{align*}
That is, $A'_1$ is $H_2$- and $H_4$-tight. However, $A'_1$ is also $H_1$-tight but, for example, $x\in A'_1-(H_1\cup H_2)$, contradicting Lemma~\ref{lem:tight}\ref{it:f}. 
\end{claimproof}

Recall that $A'_1$ and $B_1$ are $H_1$- and $H_3$-tight, while $A'_2$ and $B_2$ are $H_2$- and $H_4$-tight. Hence, by Lemma~\ref{lem:tight}\ref{it:f}, $A'_1\cap B_1\subseteq (H_1\cup H_3)-(H_2\cap H_4)$ and $A'_2\cap B_2\subseteq (H_2\cup H_4)-(H_1\cap H_3)$. By Claim~\ref{cl:z} and by symmetries on the roles of $A'_1\cap B_1$ and $A'_2\cap B_2$, of $H_1$ and $H_3$, and of $H_2$ and $H_4$, we may assume that there exists $z\in A'_1\cap B_1$ for which $z\in H_1-(H_2\cup H_3\cup H_4)$ or $z\in (H_1\cap H_3)-H_2$. Let 
\begin{gather*}
A'_1\cap B_2\cap H_{1,2}=\{e_1,\dots,e_{d/2}\},\quad A'_2\cap B_1\cap H_{1,4}=\{f_1,\dots,f_{d/2}\},\\ 
A'_1\cap B_2\cap H_{3,4}=\{g_1,\dots,g_{d/2}\},\quad A'_2\cap B_1\cap H_{2,3}=\{h_1,\dots,h_{d/2}\}.
\end{gather*}
Assume first that $z\in H_1-(H_2\cup H_3\cup H_4)$. Define 
\begin{align*}
    &A^1_1\coloneqq A'_1-z+f_1, &&A^1_2\coloneqq A'_2+z-f_1, &&\\
    &A^{2i}_1\coloneqq A^{2i-1}_1-g_i+h_i, &&A^{2i}_2\coloneqq A^{2i-1}_2+g_i-h_i && \text{for $i=1,\dots,d/2$,}\\
    &A^{2i+1}_1\coloneqq A^{2i}_1-e_i+f_{i+1}, &&A^{2i+1}_2\coloneqq A^{2i}_2+e_i-f_{i+1} && \text{for $i=1,\dots,d/2-1$,}\\
    &A^{d+1}_1\coloneqq A^{d}_1-e_{d/2}+z=B_1, &&A^{d+1}_2\coloneqq A'_2+e_{d/2}-z=B_2. &&
\end{align*}
As $z\in H_1-(H_2\cup H_3\cup H_4)$, we get that $A^i_1$ is $H_1$-tight for $i=1,\dots,d+1$, $A^{2i+1}_2$ is $H_2$-tight for $i=0,\dots,d/2$, and $A^{2i}_2$ is $H_4$-tight for $i=1,\dots,d/2$. 

Now assume that $z\in (H_1\cap H_3)-H_2$. Define
\begin{align*}
    &A^1_1\coloneqq A'_1-z+f_1, &&A^1_2\coloneqq A'_2+z-f_1, &&\\
    &A^{2i}_1\coloneqq A^{2i-1}_1-e_i+h_i, &&A^{2i}_2\coloneqq A^{2i-1}_2+e_i-h_i && \text{for $i=1,\dots,d/2$,}\\
    &A^{2i+1}_1\coloneqq A^{2i}_1-g_i+f_{i+1}, &&A^{2i+1}_2\coloneqq A^{2i}_2+g_i-f_{i+1} && \text{for $i=1,\dots,d/2-1$,}\\
    &A^{d+1}_1\coloneqq A^{d}_1-g_{d/2}+z=B_1, &&A^{d+1}_2\coloneqq A'_2+g_{d/2}-z=B_2. &&
\end{align*}
As $z\in (H_1\cap H_3)-H_2$, we get that $A^{2i+1}_1$ is $H_1$-tight for $i=0,\dots,d/2$, $A^{2i}_1$ is $H_3$-tight for $i=1,\dots,d/2$, and $A^i_2$ is $H_2$-tight for $i=1,\dots,d+1$.

Thus, in both cases, we defined a sequence of $d+1$ symmetric exchanges that transforms $(A'_1,A'_2)$ into $(B_1,B_2)$. This concludes the proof of the theorem.
\end{proof}

Assuming an independence oracle access to the matroid, the proof of Theorem~\ref{thm:main} immediately implies a polynomial-time algorithm that determines a sequence of symmetric exchanges transforming $(A_1,A_2)$ into $(B_1,B_2)$. Indeed, the sequences $\ul{x}$ and $\ul{y}$ are built up greedily. In Claim~\ref{cl:schwarcz}, one can check if the sequences can be extended by considering all possible pairs $x,x'\in A'_1\cap B_2$ and $y,y'\in A'_2\cap B_1$, and trying all the modifications discussed in Cases 1-4 of the proof of the claim. When $\ul{x}$ and $\ul{y}$ cannot be extended, then one can determine sets $X$, $Y$, $Z$ and $W$ such that $\{X,Y\}=\{A'_1\cap B_2\cap H_{1,2},A'_1\cap B_2\cap H_{3,4}\}$ and $\{Z,W\}=\{A'_2\cap B_1\cap H_{1,4},A'_2\cap B_1\cap H_{2,3}\}$ using single exchanges and the $H_1$- and $H_3$-tightness of $A'_1$ and the $H_2$- and $H_4$-tightness. Finally, checking for all $z\in (A'_1\cap B_1)\cup (A'_2\cap B_2)$ if the exchanges described at the end of the proof are feasible or not results in a desired sequence.

Now we show how to verify Conjectures~\ref{conj:gabow}-\ref{conj:equitable} for split matroids using Theorem~\ref{thm:main}. The proofs are immediate from the statement of the theorem. 

\begin{cor}\label{cor:gabow}
Conjecture~\ref{conj:gabow} holds for split matroids.
\end{cor}
\begin{proof}
Consider the ordered basis pairs $(A,B)$ and $(B,A)$; these pairs are clearly compatible. Note that $|A-B|=|A|=r$ is a lower bound for the minimum length of a sequence of symmetric exchanges that transforms $(A,B)$ into $(B,A)$. By Theorem~\ref{thm:main}, there exists such a sequence with length exactly $\min\{r,|A-B|+1\}=r$. By defining the pairs $a_i,b_i$ to be the symmetric exchanges of this sequence, the theorem follows.  
\end{proof} 
 
\begin{cor}\label{cor:white}
Conjecture~\ref{conj:white} holds for basis sequences of length two in split matroids.
\end{cor}
\begin{proof}
Theorem~\ref{thm:main} implies that if $(A_1,A_2)$ and $(B_1,B_2)$ are compatible basis sequences, then the latter can be obtained from the former by a sequence of symmetric exchanges, thus the corollary follows.
\end{proof} 

As both Corollary~\ref{cor:gabow} and Corollary~\ref{cor:white} implies the equitability of split matroids, we also get the following. 

\begin{cor}\label{cor:equitable}
Conjecture~\ref{conj:equitable} holds for split matroids.
\end{cor}

\section{Maximum length strictly monotone exchange sequences}
\label{sec:monotone}

The essence of Theorem~\ref{thm:main} is that in split matroids we can transform a basis pair $(A_1,A_2)$ into another $(B_1,B_2)$ by using at most one more symmetric exchanges than the obvious lower bound $r-|A_1\cap B_1|$. However, one might be interested in finding a longest strictly monotone exchange sequence for the basis pairs. It turns out that the sequence determined in the first half of the algorithm is longest possible.

\begin{thm} \label{thm:alg}
Let $(A_1, A_2)$ and $(B_1,B_2)$ be compatible basis pairs of a rank-$r$ split matroid $M$. Then a longest strictly monotone sequence of symmetric exchanges for $(A_1,A_2)$ and $(B_1,B_2)$ can be determined using a polynomial number of oracle calls.
\end{thm}
\begin{proof}
We follow the proof and notations of Theorem~\ref{thm:main}. Similarly to the proof there, we may assume that $M$ is connected and $A_1\cap A_2=\emptyset$. We show that if the length of $\ul{x}$ and $\ul{y}$ cannot be increased as in Claim~\ref{cl:schwarcz}, then they correspond to a longest strictly monotone sequence of symmetric exchanges. Let $\ul{x}'' = (x''_1, \dots, x''_t)$ and $\ul{y}'' = (y''_1, \dots, y''_t)$ be arbitrary sequences corresponding to a strictly monotone sequence of symmetric exchanges for $(A_1, A_2)$ and $(B_1, B_2)$. Our goal is to show that the length of $\ul{x}''$ is at most the length of $\ul{x}$.

We claim that $x''_i$ and $y''_i$ are contained in the same hyperedges from $\{H_1, H_2, H_3, H_4\}$ for each $1\le i \le t$. Indeed, assume that this holds for each pair $x''_j,y''_j$ with $1\leq j\leq i-1$. Then $(A_1-\{x''_1, \dots, x''_{i-1}\})\cup \{y''_1, \dots, y''_{i-1}\}$ is $H_1$- and $H_3$-tight and $(A_2-\{y''_1,\dots, y''_{i-1}\})\cup \{x''_1, \dots, x''_{i-1}\}$ is $H_2$- and $H_4$-tight, since $A_1$ is $H_1$- and $H_3$-tight and $A_2$ is $H_2$ and $H_4$-tight by Claim~\ref{cl:tight}. As $(A_1\cap B_2) \cup (A_2 \cap B_1) \subseteq (H_1 \triangle H_3) \cap (H_2 \triangle H_4)$ by Claims~\ref{cl:symdiff} and \ref{cl:schwarcz}, $x''_i \in H_{k,\ell}$ and $y''_i \in H_{k',\ell'}$ holds for some $k,k'\in \{1,3\}$ and $\ell,\ell'\in \{2,4\}$. The facts that $(A_1-\{x''_1, \dots, x''_i\})\cup \{y''_1, \dots, y''_i\}$ is a basis, $(A_1-\{x''_1, \dots, x''_{i-1}\})\cup \{y''_1, \dots, y''_{i-1}\}$ is $H_{k'}$-tight, and $y''_i \in H_{k'}$, together imply that $x''_i \in H_{k'}$, that is, $k=k'$. Similarly, it follows from $(A_2-\{y''_1,\dots, y''_i\})\cup \{x''_1,\dots, x''_i\}$ being a basis that $\ell=\ell'$. Therefore $x''_i$ and $y''_i$ are indeed contained in the same hyperedges from $\{H_1,H_2,H_3,H_4\}$ as claimed.

By the above, $|\{x''_1, \dots, x''_t\} \cap (H_{1,4} \cup H_{2,3})| = |\{y''_1, \dots, y''_t\} \cap (H_{1,4} \cup H_{2,3})|$. Since $|(A_2 \cap B_1) \cap (H_{1,4} \cup H_{2,3})| = |(A_1 \cap B_2) \cap (H_{1,4} \cup H_{2,3})|+d$ holds by Claims~\ref{cl:symdiff} and \ref{cl:schwarcz}, it follows that $|(A_2 \cap B_1) - \{x''_1, \dots, x''_t\}| \ge d$, thus the length of $\ul{x}''$ is at most the length of $\ul{x}$, concluding the proof of the theorem.
\end{proof}

A matroid $M$ with basis-family $\cB$ is \textbf{base orderable} if for any two bases $A,B\in\cB$ there exists a bijection $\varphi\colon A\to B$ such that $A-e+\varphi(e)\in\cB$ and $B-\varphi(e)+e \in \cB$ for every $e\in A$. Base orderable matroids are interesting and important because we have a fairly good global understanding of their structure, while frustratingly little is known about the general case. 

The proof of Theorem~\ref{thm:main} implies that there exists a `long' strictly monotone sequence of symmetric exchanges in split matroids.  

\begin{thm} \label{thm:monotone} Let $(A_1, A_2)$ and $(B_1,B_2)$ be compatible basis pairs of a rank-$r$ matroid $M$. Then there exists a strictly monotone symmetric exchange sequence of length
\begin{enumerate}[label={(\alph*)}] \itemsep0em
    \item $r-3|A_1 \cap B_1|$, if $M$ is a split matroid, \label{it:split}
    \item $r-2|A_1 \cap B_1|$, if $M$ is a base orderable split matroid, \label{it:bo}
    \item $r-|A_1 \cap B_1|-2$, if $M$ is a paving matroid. \label{it:paving}
\end{enumerate}
\end{thm}
\begin{proof}
We follow the proof and notations of Theorem~\ref{thm:main}. Similarly to the proof there, we may assume that $M$ is connected and $A_1 \cap A_2 = \emptyset$. Furthermore, assume that there is no (strictly monotone) symmetric exchange sequence of length $r-|A_1 \cap B_1|$ for $(A_1,A_2)$ and $(B_1, B_2)$ as otherwise we are done. Then $\ul{x}$ and $\ul{y}$ define a strictly monotone symmetric exchange sequence of length $r-|A_1\cap B_1|-d$, hence \ref{it:split} follows by the inequality $d\le 2|A_1 \cap B_1|$ of Claim~\ref{cl:d}.  

If $M$ is base orderable, then let $\varphi\colon A'_1 \to A_2$ be a bijection such that $A'_1-e+\varphi(e)$ and $A_2+e-\varphi(e)$ are bases for each $e\in A'_1$. It follows from the definition of $\varphi$ that it is identical on $A'_1 \cap A_2 = \{y_1, \dots, y_s\}$. Note that $A'_1 - A_2 = (A_1 \cap B_1) \cup (A'_1 \cap B_2)$ and $A_2 - A'_1 = (A_2 \cap B_2) \cup (A'_2 \cap B_1)$ where $A'_1 \cap B_2 \subseteq H_{1,2} \cup H_{3,4}$ and $A'_2 \cap B_1 \subseteq H_{1,4} \cup H_{2,3}$ by Claim~\ref{cl:symdiff}. For each $e \in A'_1 \cap H_{1,2}$, we have $f(e) \in H_2-H_3$ since $A'_1$ is $H_3$-tight and $A_2$ is $H_2$-tight by Claim~\ref{cl:tight}. Similarly, for each $e \in A'_1 \cap H_{3,4}$, we have $f(e) \in H_4-H_1$. It follows that $f(e) \in A_2 \cap B_2$ for each $e \in A'_1 \cap B_2$, thus $d = |A'_1 \cap B_2| \le |A_2 \cap B_2| = |A_1 \cap B_1|$, proving \ref{it:bo}.

If $M$ is paving, then the $H_3$-tightness of $A'_1$ implies $|A'_1\cap H_3|=r-1$. By Claims~\ref{cl:symdiff} and \ref{cl:tight}, we get $d/2 = |(A'_1 \cap B_2) \cap H_{1,2}| \le |A'_1- H_3| = 1$, proving \ref{it:paving}.
\end{proof}

\section{Conclusion}

In this paper, we studied the exchange distance of basis pairs in split matroids. We proved that for any two compatible basis pairs, there exists a sequence of symmetric exchanges that transforms one into the other, and that the length of this sequence is at most one more than the trivial lower bound. The proof was algorithmic, hence the sequence of exchange in question can be determined in polynomial time. As a corollary, we verified several long-standing conjectures for the class of split matroids, and also provided lower bounds on the maximum length of strictly monotone symmetric exchange sequences.

A lesser known conjecture of Hamidoune states that the exchange distance of compatible basis pairs is at most the rank of the matroid, see~\cite{cordovil1993bases}. However, until now, the conjecture was open even for sparse paving matroids. The statement might seem rather optimistic at first glance, still, our result settles it for split matroids. It remains an intriguing open question whether the conjecture holds in general; such a result would immediately imply Conjecture~\ref{conj:gabow}, Conjecture~\ref{conj:white} for sequences of length two and Conjecture~\ref{conj:equitable}.

\section*{Acknowledgement} The work was supported by the Lend\"ulet Programme of the Hungarian Academy of Sciences -- grant number LP2021-1/2021 and by the Hungarian National Research, Development and Innovation Office -- NKFIH, grant numbers FK128673 and TKP2020-NKA-06.

\bibliographystyle{abbrv}
\bibliography{split_conjectures}

\begin{thebibliography}{10}

\bibitem{egres_open_equit}
Equitability of matroids.
\newblock \url{http://lemon.cs.elte.hu/egres/open/Equitability_of_matroids}.
\newblock Accessed: 2021-01-28.

\bibitem{aharoni2017faira}
R.~Aharoni, E.~Berger, D.~Kotlar, and R.~Ziv.
\newblock Fair representation in the intersection of two matroids.
\newblock {\em The Electronic Journal of Combinatorics}, 24(4):P4.10, 2017.

\bibitem{berczi2022hypergraph}
K.~B\'erczi, T.~Kir\'aly, T.~Schwarcz, Y.~Yamaguchi, and Y.~Yokoi.
\newblock Hypergraph characterization of split matroids.
\newblock {\em arXiv preprint arXiv:2202.04371}, 2022.

\bibitem{blasiak2008toric}
J.~Blasiak.
\newblock The toric ideal of a graphic matroid is generated by quadrics.
\newblock {\em Combinatorica}, 28(3):283--297, 2008.

\bibitem{bonin2013basis}
J.~E. Bonin.
\newblock Basis-exchange properties of sparse paving matroids.
\newblock {\em Advances in Applied Mathematics}, 50(1):6--15, 2013.

\bibitem{cameron2021excluded}
A.~Cameron and D.~Mayhew.
\newblock Excluded minors for the class of split matroids.
\newblock {\em Australasian Journal of Combinatorics}, 79(2):195--204, 2021.

\bibitem{cordovil1993bases}
R.~Cordovil and M.~L. Moreira.
\newblock Bases-cobases graphs and polytopes of matroids.
\newblock {\em Combinatorica}, 13(2):157--165, 1993.

\bibitem{farber1989basis}
M.~Farber.
\newblock Basis pair graphs of transversal matroids are connected.
\newblock {\em Discrete mathematics}, 73(3):245--248, 1989.

\bibitem{farber1985edge}
M.~Farber, B.~Richter, and H.~Shank.
\newblock Edge-disjoint spanning trees: A connectedness theorem.
\newblock {\em Journal of Graph Theory}, 9(3):319--324, 1985.

\bibitem{fekete2011equitable}
Z.~Fekete and J.~Szab{\'{o}}.
\newblock Equitable partitions to spanning trees in a graph.
\newblock {\em The Electronic Journal of Combinatorics}, 18(1), 2011.

\bibitem{frank2011connections}
A.~Frank.
\newblock {\em Connections in Combinatorial Optimization}, volume~38 of {\em
  Oxford Lecture Series in Mathematics and its Applications}.
\newblock Oxford University Press, Oxford, 2011.

\bibitem{gabow1976decomposing}
H.~Gabow.
\newblock Decomposing symmetric exchanges in matroid bases.
\newblock {\em Mathematical Programming}, 10(1):271--276, 1976.

\bibitem{greene1975some}
C.~Greene and T.~L. Magnanti.
\newblock Some abstract pivot algorithms.
\newblock {\em SIAM Journal on Applied Mathematics}, 29(3):530--539, 1975.

\bibitem{hartmanis1959lattice}
J.~Hartmanis.
\newblock Lattice theory of generalized partitions.
\newblock {\em Canadian Journal of Mathematics}, 11:97--106, 1959.

\bibitem{joswig2017matroids}
M.~Joswig and B.~Schr{\"o}ter.
\newblock Matroids from hypersimplex splits.
\newblock {\em Journal of Combinatorial Theory, Series A}, 151:254--284, 2017.

\bibitem{kajitani1988ordering}
Y.~Kajitani, S.~Ueno, and H.~Miyano.
\newblock Ordering of the elements of a matroid such that its consecutive w
  elements are independent.
\newblock {\em Discrete Mathematics}, 72(1-3):187--194, 1988.

\bibitem{kashiwabara2010toric}
K.~Kashiwabara.
\newblock The toric ideal of a matroid of rank 3 is generated by quadrics.
\newblock {\em The Electronic Journal of Combinatorics}, 17(1):R28, 2010.

\bibitem{kotlar2013circuits}
D.~Kotlar.
\newblock On circuits and serial symmetric basis-exchange in matroids.
\newblock {\em SIAM Journal on Discrete Mathematics}, 27(3):1274--1286, 2013.

\bibitem{kotlar2021sequential}
D.~Kotlar, E.~Roda, and R.~Ziv.
\newblock On sequential basis exchange in matroids.
\newblock {\em SIAM Journal on Discrete Mathematics}, 35(4):2517--2519, 2021.

\bibitem{kotlar2013serial}
D.~Kotlar and R.~Ziv.
\newblock On serial symmetric exchanges of matroid bases.
\newblock {\em Journal of Graph Theory}, 73(3):296--304, 2013.

\bibitem{lason2014toric}
M.~Laso{\'n} and M.~Micha{\l}ek.
\newblock On the toric ideal of a matroid.
\newblock {\em Advances in Mathematics}, 259:1--12, 2014.

\bibitem{mcguinness2020frame}
S.~McGuinness.
\newblock Frame matroids, toric ideals, and a conjecture of {W}hite.
\newblock {\em Advances in Applied Mathematics}, 118:102042, 2020.

\bibitem{oxley2011matroid}
J.~Oxley.
\newblock {\em Matroid Theory}, volume~21 of {\em Oxford Graduate Texts in
  Mathematics}.
\newblock Oxford University Press, Oxford, second edition, 2011.

\bibitem{schweig2011toric}
J.~Schweig.
\newblock Toric ideals of lattice path matroids and polymatroids.
\newblock {\em Journal of Pure and Applied Algebra}, 215(11):2660--2665, 2011.

\bibitem{van2012cyclic}
J.~van~den Heuvel and S.~Thomass{\'e}.
\newblock Cyclic orderings and cyclic arboricity of matroids.
\newblock {\em Journal of Combinatorial Theory, Series B}, 102(3):638--646,
  2012.

\bibitem{welsh1976matroid}
D.~J.~A. Welsh.
\newblock {\em Matroid Theory}.
\newblock Academic Press, London, 1976.

\bibitem{white1980unique}
N.~L. White.
\newblock A unique exchange property for bases.
\newblock {\em Linear Algebra and its Applications}, 31:81--91, 1980.

\bibitem{wiedemann1984cyclic}
D.~Wiedemann.
\newblock Cyclic base orders of matroids.
\newblock Manuscript, 1984.

\end{thebibliography}

\end{document}